\documentclass{amsart}

\usepackage{amssymb}
\usepackage[numeric, abbrev, nobysame]{amsrefs}
\usepackage{amscd}

\def\frak{\mathfrak}
\def\Bbb{\mathbb}
\def\Cal{\mathcal}

\newtheorem{thm}{Theorem}
\newtheorem*{thm*}{Theorem}
\newtheorem{prop}[thm]{Proposition}
\newtheorem*{prop*}{Proposition}

\newtheorem*{lem*}{Lemma}
\newtheorem{kor}[thm]{Corollary}
\newtheorem*{kor*}{Corollary}

\newcommand{\Ad}{\operatorname{Ad}}
\renewcommand{\exp}{\operatorname{exp}}
\newcommand{\id}{\operatorname{id}}

\newcommand{\End}{\operatorname{End}}

\newcommand{\fg}{{\frak g}}
\newcommand{\x}{\times}
\renewcommand{\o}{\circ}

\let\ccdot\cdot
\def\cdot{\hbox to 2.5pt{\hss$\ccdot$\hss}}

\newcommand{\al}{\alpha}
\newcommand{\be}{\beta}
\newcommand{\ga}{\gamma}

\newcommand{\la}{\lambda}
\newcommand{\om}{\omega}
\renewcommand{\phi}{\varphi}
\newcommand{\ph}{\varphi}
\newcommand{\ps}{\psi}
\newcommand{\si}{\sigma}
\renewcommand{\th}{\theta}

\newcommand{\Ga}{\Gamma}
\newcommand{\La}{\Lambda}
\newcommand{\Om}{\Omega}
\newcommand{\Ph}{\Phi}
\newcommand{\Ps}{\Psi}

\def\sideremark#1{\ifvmode\leavevmode\fi\vadjust{\vbox to0pt{\vss% the remark
 \hbox to 0pt{\hskip\hsize\hskip1em%                          will appear only
 \vbox{\hsize3cm\tiny\raggedright\pretolerance10000%          on the side
 \noindent #1\hfill}\hss}\vbox to8pt{\vfil}\vss}}}%
\begin{document}
\title{Equivariant quantizations for AHS--structures}

\author{Andreas \v Cap and Josef \v Silhan}

\address{AC: Institut f\"ur Mathematik, Universit\"at Wien, Nordbergstra\ss
  e 15, A--1090 Wien, Austria; JS: Max Planck Institute
  Institute for Mathematics, Vivatsgasse 7, 53111 Bonn,
  Germany} 

\email{Andreas.Cap@esi.ac.at, silhan@math.muni.cz} 

\thanks{First author supported by project P 19500--N13 of the ``Fonds
  zur F\"orderung der wissenschaftlichen Forschung'' (FWF), second
  author was supported by a Junior Fellowship of the Erwin Schr\"odinger
  Institute (ESI)}


\date{April 20, 2009}


\begin{abstract}
  We construct an explicit scheme to associate to any potential symbol
  an operator acting between sections of natural bundles (associated
  to irreducible representations) for a so--called AHS--structure.
  Outside of a finite set of critical (or resonant) weights, this
  procedure gives rise to a quantization, which is intrinsic to this
  geometric structure. In particular, this provides projectively and
  conformally equivariant quantizations for arbitrary symbols on
  general (curved) projective and conformal structures.
\end{abstract}

\maketitle

\section{Introduction}\label{1}

Consider a smooth manifold $M$, two vector bundles $E$ and $F$ over
$M$ and a linear differential operator $D:\Ga(E)\to\Ga(F)$, where
$\Ga(\ )$ indicates the space of smooth sections. If $D$ is of order
at most $k$, then it has a well defined ($k$th order)
\textit{principal symbol} $\si_D$, which can be viewed as a vector
bundle map $S^kT^*M\otimes E\to F$ or as a smooth section of the vector
bundle $S^kTM\otimes E^*\otimes F$. Here $TM$ and $T^*M$ are the
tangent respectively cotangent bundle of $M$, $E^*$ is the bundle dual
to $E$, and $S^k$ denotes the $k$th symmetric power. 

A \textit{quantization} on $M$ is a right inverse to the principal
symbol map. This means that to each smooth section $\tau$ of the
bundle $S^kTM\otimes E^*\otimes F$, one has to associate a
differential operator $A_\tau:\Ga(E)\to\Ga(F)$ of order $k$ with
principal symbol $\tau$. Note that operators of order $0$ coincide
with their principal symbols, so there a unique possible quantization
in order $0$. Given any $k$th order operator $D$ with principal symbol
$\tau$, the difference $D-A_\tau$ is of order $k-1$. Iterating this, we
conclude that, having a quantization in each order $\leq k$, one
actually obtains an isomorphism between the space $Diff^k(E,F)$ of
differential operators $\Ga(E)\to\Ga(F)$ of order at most $k$ and the
space of smooth sections of the bundle $\oplus_{i=0}^kS^iTM\otimes
E^*\otimes F$.

A classical example of a quantization is provided by the Fourier
transform for smooth functions on $\Bbb R^n$. However, it is well
known that (even for $E=F=M\x\Bbb R$) there is no canonical
quantization on a general manifold $M$, but one has to make additional
choices. For our purposes, the most relevant example is to choose
linear connections on the vector bundles $E$ and $TM$. Having done
this, one obtains induced linear connections on duals and tensor
products of these bundles, and we will denote all these connections by
$\nabla$.  For a smooth section $s$ of $E$, one can then form the
$k$--fold covariant derivative $\nabla^k s$, which is a section of
$\otimes^k T^*M\otimes E$. Symmetrizing in the $T^*M$ entries, we
obtain a section $\nabla^{(k)}s$ of $S^kT^*M\otimes E$.
Viewing a symbol $\tau$ as a bundle map $S^kT^*M\otimes E\to F$, we
can simply put $A_\tau(s):=\tau(\nabla^{(k)}s)$. Clearly this defines
a differential operator $A_\tau$ of order $k$ and it is well known
that its principal symbol is $\tau$, so we have obtained a
quantization in this way.

This provides a link to geometry. Suppose that $M$ is endowed with
some geometric structure which admits a canonical connection. Then one
obtains quantizations for all natural bundles associated to this
structure. The classical example of this situation is the case when
$(M,g)$ is a Riemannian manifold. Then the natural bundles are tensor
and spinor bundles, and on each such bundle one has the Levi--Civita
connection. Hence the above procedure leads to a natural quantization
(in the sense that it is intrinsic to the Riemannian structure) for
any pair $E$ and $F$ of natural vector bundles.

At this point there arises the question whether weaker geometric
structures, which do not admit canonical connections, still do admit
natural quantizations. This problem has been originally posed in
\cite{LO} and has been intensively studied since then. The examples
above naturally lead to the two geometric structures for which this
problem has been mainly considered. On the one hand, one may replace a
single linear connection on $TM$ by a projective equivalence class of
such connections. Here two connections are considered as equivalent if
they have the same geodesics up to parametrization. On the other hand,
the most natural weakening of Riemannian metrics is provided by
conformal structures.  Here one takes an equivalence class of
(pseudo--)Riemannian metrics which are obtained from each other by
multiplication by positive smooth functions.

Projective and conformal structures fit into the general scheme of
so--called AHS--structures. These are geometric structures which admit
an equivalent description by a canonical Cartan connection modelled on
a compact Hermitian symmetric space $G/P$, where $G$ is semisimple and
$P\subset G$ is an appropriate parabolic subgroup. These geometries
and the more general class of parabolic geometries have been studied
intensively during the last years, and several striking results have
been obtained, see e.g.~\cite{CSS-BGG}. In particular, an efficient
differential calculus for these structures based on so--called tractor
bundles has been worked out in \cite{tractors}.

This general point of view has shown up in the theory of equivariant
quantizations already. Namely, it turns out that the homogeneous space
$G/P$ always contains a dense open subset (the big Schubert--cell)
which is naturally diffeomorphic to $\Bbb R^n$.  While the $G$--action
on $G/P$ cannot be restricted to this subspace, one obtains a
realization of the Lie algebra $\fg$ of $G$ as a Lie algebra of vector
fields on $\Bbb R^n$. For the homogeneous model $G/P$ and geometries
locally isomorphic to it, naturality of a quantization is then
equivalent to equivariancy for the action of this Lie algebra of
vector fields. In many articles, the question of quantizations
naturally associated to a projective and/or conformal structure is
posed in this setting. Also, the algebras corresponding to general
AHS--structures have been studied in this setting under the name
``IFFT--equivariant quantizations'', see \cite{IFFT}. It should be
pointed out however, that these methods only apply to geometries
locally isomorphic to $G/P$ (e.g.~to locally conformally flat
conformal structures). As it is well known from the theory of linear
invariant differential operators, passing from the locally flat
category to general structures is a very difficult problem.

Most of the work on natural quantizations only applies to operators on
sections of line bundles (density bundles). It was only recently that
the methods for projective structures have been extended to general
natural vector bundles in \cite{Hansoul}. The construction there uses
the Thomas--Whitehead (or ambient) description of projective
structures, which is an equivalent encoding of the canonical Cartan
connection for projective structures. This approach is only available
in the projective case, though. As mentioned in \cite{Hansoul}, there
is hope to use the Fefferman--Graham ambient metric for conformal
structures to find conformally invariant quantizations, but there
several immediate problems with this approach. For the other
AHS--structures, there is no clear analog of the ambient description.

It should be also mentioned that the results for projective structures
have been obtained using the canonical Cartan connection, see
\cite{Mathonet-Radoux}. After this article was essentially completed,
we learned about the recent preprint \cite{Mathonet-Radoux2}, in which
the Cartan approach is extended to prove existence of a natural
quantization for conformal structures and it is claimed that the
method further extends to all AHS--structures.

In this article, we use the recent advances on invariant calculi for
parabolic geometries to develop a scheme for constructing equivariant
quantizations. This scheme is explicit and uniform, it applies to all
AHS--structures and to all (irreducible) natural bundles for such
structures. As it is known from the special cases studied so far,
equivariant quantizations do not always exist, so our scheme does not
always lead to an equivariant quantization. 

To formulate the result more precisely, we need a bit more background.
It turns out that for any AHS--structure there is a family of natural
line bundles $\Cal E[w]$ parametrized by a real number $w$, the
so--called density bundles. Any natural bundle $E$ can be twisted by
forming tensor products with density bundles to obtain bundles
$E[w]:=E\otimes\Cal E[w]$. (For conformal structures, this free
parameter is known as ``conformal weight''.) Doing this to the target
bundle of differential operators, we can view a section
$\tau\in\Ga(S^kTM\otimes E^*\otimes F\otimes\Cal E[\delta])$ as the
potential symbol of an operator $\Ga(E)\to\Ga(F[\delta])$. We first
universally decompose the bundle of symbols into a finite direct sum
of subbundles. On the level of sections, we write this decomposition
as $\tau=\sum_i\tau_i$. Given such a section, our scheme constructs a
differential operator $A_\tau:\Ga(E)\to\Ga(F[\delta])$ for any choice
of weight $\delta$. The principal symbol of $A_\tau$ is
$\sum_i\ga_i\tau_i$ for real numbers $\ga_i$ which only depend on $i$,
and $\delta$ (and not on $\tau$ or on the manifold in question). We
prove that each $\ga_i$ is nonzero except for finitely many values of
$\delta$. Whenever all $\ga_i$ are non--zero, we obtain a natural
quantization by mapping $\tau$ to $A_{\sum_i\ga_i^{-1}\tau_i}$.

Our method does not only lead to an abstract proof that the set of
critical weights (i.e.~of weights $\delta$ for which some $\ga_i$
vanishes) is finite. We also get general information on the number and
size of critical weights. In each concrete example, one can determine
the set of critical weights explicitly, and this needs only finite
dimensional representation theory.

We should mention that the developments in this article are closely
related to the results in the recent thesis \cite{Kroeske} of
J.~Kroeske, in which the author systematically constructs bilinear
natural differential operators for AHS--structures and, more
generally, for parabolic geometries.

\section{AHS--structures and invariant calculus}\label{2}
In this section we review basics facts on AHS--structures and
invariant differential calculus for these geometries. Our basic
references are \cite{Yamaguchi}, \cite{Weyl}, and \cite{CSS2}.

\subsection{$|1|$--graded Lie algebras and first order structures}\label{2.1}
The starting point for defining an AHS--structure is a simple Lie
algebra $\fg$ endowed with a so called $|1|$--grading, i.e.~a
decomposition $\fg=\fg_{-1}\oplus\fg_0\oplus\fg_1$, such that
$[\fg_i,\fg_j]\subset\fg_{i+j}$, where we agree that $\fg_\ell=0$ for
$\ell\notin\{-1,0,1\}$. The classification of such gradings is well
known, since it is equivalent to the classification of Hermitian
symmetric spaces. We put $\frak p:=\frak g_0\oplus\fg_1\subset\fg$. By
the grading property, $\frak p$ is a subalgebra of $\frak p$ and
$\fg_1$ is a nilpotent ideal in $\frak p$.

Given a Lie group $G$ with Lie algebra $\fg$, there are natural
subgroups $G_0\subset P\subset G$ corresponding to the Lie subalgebras
$\fg_0\subset\frak p\subset\fg$. For $P$ one may take a subgroup lying
between the normalizer $N_G(\frak p)$ of $\frak p$ in $G$ and its
connected component of the identity. Then $G_0\subset P$ is defined as
the subgroup of all elements whose adjoint action preserves the
grading of $\fg$. In particular, restricting the adjoint action to
$\fg_{-1}$, one obtains a representation $G_0\to GL(\fg_{-1})$.  This
representation is infinitesimally injective, so it makes sense to talk
about first order G--structures with structure group $G_0$ on smooth
manifolds of dimension $\dim(\fg_{-1})$.

By definition, such a structure is given by a smooth principal bundle
$p:\Cal G_0\to M$ with structure group $G_0$, such that the associated
bundle $\Cal G_0\x_{G_0}\fg_{-1}$ is isomorphic to the tangent bundle
$TM$. It turns out that the Killing form on $\fg$ induces a
$G_0$--equivariant duality between $\fg_{-1}$ and $\fg_1$, so $\Cal
G_0\x_{G_0}\fg_1\cong T^*M$. Using this, one can realize arbitrary
tensor bundles on $M$ as associated bundles to $\Cal G_0$. More
generally, any representation of $G_0$, via forming associated
bundles, gives rise to a natural vector bundle on manifolds endowed
with such a structure. It turns out that $G_0$ is always reductive
with one--dimensional center. Hence finite dimensional representations
of $G_0$ on which the center acts diagonalizably (which we will always
assume in the sequel) are completely reducible, i.e.~they split into
direct sums of irreducible representations.

The one--dimensional center of $G_0$ leads to a family of natural line
bundles. For $w\in\Bbb R$, we can define a homomorphism $G_0\to\Bbb
R_+$ by mapping $g\in G_0$ to $|\det(\Ad_-(g))|^{\tfrac{w}{n}}$, where
$n=\dim(\fg_{-1})$ and $\Ad_-(g):\fg_{-1}\to\fg_{-1}$ is the
restriction of the adjoint action of $g$. This evidently is a smooth
homomorphism, thus giving rise to a one--dimensional representation
$\Bbb R[w]$ of $G_0$. It is easy to see that $\Bbb R[w]$ is
non--trivial for $w\neq 0$. (The factor $\tfrac{1}{n}$ is included to
get the usual normalization in the case of conformal structures.) The
corresponding associated bundle will be denoted by $\Cal E[w]$, and
adding the symbol $[w]$ to the name of a natural bundle will always
indicate a tensor product with $\Cal E[w]$. Using the convention that
$1$--densities are the objects which can be naturally integrated on
non--orientable manifolds, $\Cal E[w]$ is by construction the bundle
of $(-\tfrac{w}{n})$--densities. In particular, all the bundles $\Cal
E[w]$ are trivial line bundles, but there is no canonical
trivialization for $w\neq 0$.

\subsection{Canonical Cartan connections and AHS--structures}\label{2.2}
The exponential mapping restricts to a diffeomorphism from $\fg_1$
onto a closed normal Abelian subgroup $P_+\subset P$ such that $P$ is
the semidirect product of $G_0$ and $P_+$. Hence $G_0$ can also
naturally be viewed as a quotient of $P$.  In particular, given a
principal $P$--bundle $\Cal G\to M$, the subgroup $P_+$ acts freely on
$\Cal G$, and the quotient $\Cal G/P_+$ is naturally a principal
bundle with structure group $G_0$. Next, suppose that there is a
Cartan connection $\om\in\Om^1(\Cal G,\frak g)$ on the principal
bundle $\Cal G$. Then the $\fg_{-1}$--component of $\om$ descends to a
well defined one--form $\th\in\Om^1(\Cal G/P_+,\fg_{-1})$, which is
$G_0$--equivariant and strictly horizontal.  This means that $(\Cal
G/P_+\to M,\th)$ is a first order structure with structure group
$G_0$. In this sense, any Cartan geometry $(p:\Cal G\to M,\om)$ of
type $(G,P)$ has an underlying first order structure with structure
group $G_0$. Conversely, one can talk about extending a first order
structure to a Cartan geometry.

It turns out (see e.g.~\cite{CSS2}) that, for almost all choices of
$(G,P)$, for any given first order structure with structure group
$G_0$ there is a unique (up to isomorphism) extension to a Cartan
geometry of type $(G,P)$, for which the Cartan connection $\om$
satisfies a certain normalization condition. This is usually phrased
as saying that such structures admit a canonical Cartan connection.
The main exception is $\fg=\frak{gl}(n+1,\Bbb R)$ with a
$|1|$--grading such that $\fg_0=\frak{gl}(n,\Bbb R)$ and $\fg_{\pm
  1}\cong\Bbb R^n$. For an appropriate choice of $G$, the adjoint
action identifies $G_0$ with $GL(\fg_{-1})=GL(n,\Bbb R)$. A first
order structure for this group on a manifold $M$ is just the full
linear frame bundle of $M$ and hence contains no information.  In this
case, an extension to a normal Cartan geometry of type $(G,P)$ is
equivalent to the choice of a projective equivalence class of torsion
free connections on the tangent bundle $TM$, i.e.~to a classical
projective structure.

Normal Cartan geometries of type $(G,P)$ as well as the equivalent
underlying structures (i.e.~classical projective structures
respectively first order structures with structure group $G_0$) are
often referred to as \textit{AHS--structures}. AHS is short for
``almost Hermitian symmetric''. To explain this name, recall that the
basic example of a Cartan geometry of type $(G,P)$ is provided by the
natural projection $G\to G/P$ and the left Maurer--Cartan form as the
Cartan connection. This is called the \textit{homogeneous model} of
geometries of type $(G,P)$. Now the the homogeneous spaces $G/P$ for
pairs $(G,P)$ coming from $|1|$--gradings as described above, are
exactly the compact irreducible Hermitian symmetric spaces.

\subsection{Natural bundles and the fundamental derivative}\label{2.3}
Via forming associated bundles, any representation of the group $P$
gives rise to a natural bundle for Cartan geometries of type
$(G,P)$. As we have seen above, $P$ is the semi--direct product of the
reductive subgroup $G_0$ and the normal vector subgroup $P_+$, so its
representation theory is fairly complicated. Via the quotient
homomorphism $P\to G_0$, any representation of $G_0$ gives rise to a
representation of $P$. It turns out that the representations of $P$
obtained in this way are exactly the completely reducible
representations, i.e.~the direct sums of irreducible
representations. Correspondingly, we will talk about completely
reducible and irreducible natural bundles on Cartan geometries of type
$(G,P)$. If we have a Cartan geometry $(p:\Cal G\to M,\om)$ with
underlying structure $(p_0:\Cal G_0\to M,\th)$ and $V$ is a
representation of $G_0$, which we also view as a representation of
$P$, then we can naturally identify $\Cal G\x_P V$ with $\Cal
G_0\x_{G_0}V$. Hence completely reducible bundles can be easily
described in terms of the underlying structure. 

There is a second simple source of representations of $P$, which leads
to an important class of natural bundles. Namely, one may restrict any
representation of $G$ to the subgroup $P$. The corresponding natural
vector bundles are called \textit{tractor bundles}, their general
theory is developed in \cite{tractors}. The most important tractor
bundle is the \textit{adjoint tractor bundle}. For a Cartan geometry
$(p:\Cal G\to M,\om)$ it is defined by $\Cal AM:=\Cal G\x_P\fg$, so it
is the associated bundle with respect to the restriction of the
adjoint representation of $G$ to $P$. Now the $P$--invariant subspaces
$\fg_1\subset\frak p\subset\fg$ give rise to a filtration $\Cal
A^1M\subset\Cal A^0M\subset\Cal AM$ of the adjoint tractor bundle by
smooth subbundles. By construction, $\Cal A^1M\cong T^*M$ and since
$\frak g/\frak p\cong\fg_{-1}$ we see that $\Cal AM/\Cal A^0M\cong
TM$. We will write $\Pi:\Cal AM\to TM$ for the resulting natural
projection. Hence the adjoint tractor bundle has the cotangent bundle
as a natural subbundle and the tangent bundle as a natural quotient.

The Killing form defines a $G$--invariant, non--degenerate bilinear
form on $\fg$. It turns out that $\fg_1$ is the annihilator of $\frak
p$ with respect to the Killing form, which leads to duality with
$\fg/\frak p\cong\fg_{-1}$ observed above. On the level of associated
bundles, we obtain a natural non--degenerate bilinear form on the
adjoint tractor bundle $\Cal AM$, which thus can be identified with
the dual bundle $\Cal A^*M$. Under this pairing, the subbundle $\Cal
A^1M$ is the annihilator of $\Cal A^0M$. The resulting duality between
$\Cal A^1M$ and $\Cal AM/\Cal A^0M$ is exactly the duality between
$T^*M$ and $TM$. 

The adjoint tractor bundle gives rise to a basic family of natural
differential operators for AHS--structures (and more generally for
parabolic geometries). These have been introduced in \cite{tractors}
under the name ``fundamental $D$--operators'', more recently, the name
\textit{fundamental derivative} is commonly used. Let us start with an
arbitrary representation $V$ of $P$ and consider the corresponding
natural bundle $E:=\Cal G\x_PV\to M$ for a geometry $(p:\Cal G\to
M,\om)$.  Then smooth sections of this bundle are in bijective
correspondence with smooth maps $f:\Cal G\to V$, which are
$P$--equivariant. In the special case $V=\fg$ of the adjoint tractor
bundle, we can then use the trivialization of $T\Cal G$ provided by
the Cartan connection $\om$ to identify $P$--equivariant functions
$\Cal G\to\fg$ with $P$--invariant vector fields on $\Cal G$. For a
section $s\in\Ga(\Cal AM)$, we can form the corresponding vector field
$\xi\in\frak X(\Cal G)$ and use it to differentiate the equivariant
function $f:\Cal G\to V$ corresponding to a section $\si\in\Ga(E)$.
The result will again be equivariant, thus defining a smooth section
$D_s\si\in\Ga(E)$.  Hence we can view the fundamental derivative as an
operator $D=D^E:\Ga(\Cal AM)\x\Ga(E)\to\Ga(E)$. The basic properties
of this operator as proved in section 3 of \cite{tractors} are:

\begin{prop}\label{P2.3}
Let $V$ be a representation of $P$ and let $E=\Cal G\x_PV$ be the
corresponding natural bundle for an AHS--structure $(p:\Cal G\to
M,\om)$. Then we have:

\noindent
(1) $D:\Ga(\Cal AM)\x\Ga(E)\to\Ga(E)$ is a first order differential
operator which is natural, i.e.~intrisic to the AHS--structure on $M$.

\noindent
(2) $D$ is linear over smooth functions in the $\Cal AM$--entry, so we
    can also view $\si\mapsto D\si$ as an operator $\Ga(E)\to\Ga(\Cal
    A^*M\otimes E)$.

\noindent
(3) For $s\in\Ga(\Cal AM)$, $\si\in\Ga(E)$, and $f\in C^\infty(M,\Bbb
R)$, we have the Leibniz rule $D_s(f\si)=(\Pi(s)\cdot f)\si+fD_s\si$,
where $\Pi:\Ga(\Cal AM)\to\Ga(TM)$ is the natural tensorial
projection. 

\noindent
(4) For a second natural bundle $F=\Cal G\x_PW$, a $P$--equivariant
map $V\to W$, and the corresponding bundle map $\Ph:E\to F$, the
fundamental derivatives on $E$ and $F$ are related by
$D^F_s(\Ph\o \si)=\Ph\o D^E_s\si$ for all $s\in\Ga(\Cal AM)$ and
$\si\in\Ga(E)$. 
\end{prop}

The naturality statement in (4) justifies denoting the fundamental
derivatives on all natural bundles by the same letter. Since there is
no restriction on the bundle $E$, the fundamental derivative in the
form of part (2) can evidently be iterated. For $\si\in\Ga(E)$ we can
form $D\si$, $D^2\si=D(D\si)$ and inductively
$D^k\si\in\Ga(\otimes^k\Cal A^*M\otimes E)$.

\subsection{Curved Casimir operators}\label{2.4} 

Curved Casimir operators form another basic set of natural
differential operators defined on AHS--structures. They have been
introduced in \cite{Casimir} in the general context of parabolic
geometries. That article contains all the facts about curved Casimir
operators we will need, as well as the general construction for
splitting operators that we will use below.

As above, we start with a representation $V$ of $P$ and consider the
corresponding natural vector bundle $E=\Cal G\x_PV$ for an
AHS--structure $(p:\Cal G\to M,\om)$. As noticed above, the
composition of two fundamental derivatives defines an operator
$D^2:\Ga(E)\to\Ga(\otimes^2\Cal A^*M\otimes E)$. From \ref{2.3} we
know that the Killing form on $\fg$ induces a non--degenerate bilinear
form on $\Cal AM$. Using this to identify $\Cal AM$ with $\Cal A^*M$,
we also get a natural bilinear form $B$ on $\Cal A^*M$.  This can be
used to define a bundle map $B\otimes\id:\otimes^2\Cal A^*M\otimes
E\to E$. Now one defines the \textit{curved Casimir operator} $\Cal
C=\Cal C^E:\Ga(E)\to\Ga(E)$ by $\Cal C(\si):=(B\otimes\id)\o D^2\si$.

Part (4) of proposition \ref{P2.3} easily implies (compare with
proposition 2 of \cite{Casimir}) that for another natural vector
bundle $F$ and a bundle map $\Ph:E\to F$ coming form a
$P$--equivariant map between the inducing representations, one gets
$\Cal C^F(\Ph\o\si)=\Ph\o\Cal C^E(\si)$. This is the justification for
denoting all curved Casimir operators by the same symbol.

From the construction it is clear that $\Cal C$ is a natural
differential operator of order at most $2$. However, it turns out that
$\Cal C$ actually always is of order at most one. Moreover, on
sections of bundles induced by irreducible representations, the
operator $\Cal C$ acts by a scalar which can be computed from
representation theory data. One can associate to any irreducible
representation of $\fg_0$ a highest and a lowest weight by passing to
complexifications, see section 3.4 of \cite{Casimir}. The weights are
functionals on the Cartan subalgebra $\frak h$ of the complexification
$\fg_{\Bbb C}$ of $\fg$, which at the same time is a Cartan subalgebra
for $(\fg_0)_{\Bbb C}$. Recall that the Killing form of $\fg$ induces
a positive definite inner product on the real space of functionals on
$\frak h$ spanned by possible weights for finite dimensional
representations. Denoting this inner product by $\langle\ ,\ \rangle$
and the corresponding norm by $\|\ \|$, the following result is proved
as theorem 1 in \cite{Casimir}.

\begin{prop}\label{P2.4}
  Let $V$ be a representation of $P$ and let $E=\Cal G\x_PV$ be the
  corresponding natural vector bundle for an AHS--structure $(p:\Cal
  G\to M,\om)$. Then

\noindent
(1) $\Cal C:\Ga(E)\to\Ga(E)$ is a natural differential operator of
    order at most one.

\noindent
(2) If the representation $V$ is irreducible of lowest weight $-\la$,
then $\Cal C$ acts on $\Ga(E)$ by multiplication by
$\|\la\|^2+2\langle\la,\rho\rangle$, where $\rho$ is half the sum of
all positive roots of $\fg_\Bbb C$.
\end{prop}

\section{The quantization scheme}\label{3}
Throughout this section, we fix a pair $(G,P)$, two irreducible
representations $V$ and $W$ of $G_0$ with corresponding natural
bundles $E$ and $F$, as well as an order $k>0$. Given these data, we
try to construct a quantization for $k$th order symbols of operators
mapping sections of $E$ to sections of $F[\delta]$ for $\delta\in\Bbb
R$.

The basic idea for the construction is very simple. The bundle of
symbols in this situation is $S^kTM\otimes E^*\otimes F[\delta]$. We
know from \ref{2.3} that $TM$ naturally is a quotient of the adjoint
tractor bundle $\Cal AM$, so the bundle of symbols is a quotient of
$S^k\Cal AM\otimes E^*\otimes F[\delta]$. Using the general machinery
of splitting operators, we can associate to a symbol a section of the
latter bundle. But such a section can be interpreted as a bundle map
$S^k\Cal A^*M\otimes E\to F[\delta]$, so we can apply it to the values
of the symmetrized $k$--fold fundamental derivative of sections of
$E$. 

\subsection{Some properties of the fundamental derivative}\label{3.1} 
To carry out this idea, we first have to derive some properties of the
iterated fundamental derivative $D^k$ and its symmetrization
$D^{(k)}:\Ga(E)\to\Ga(\Cal WM)$, where $\Cal WM:=S^k\Cal A^*M\otimes
E$.  Recall from \ref{2.3} that $\Cal AM$ admits a natural filtration
of the form $\Cal A^1M\subset\Cal A^0M\subset\Cal A^{-1}M:=\Cal AM$.
Since elements of $\Cal WM$ can be interpreted as $k$--linear,
symmetric maps $(\Cal AM)^k\to E$, we get an induced filtration of the
bundle $\Cal WM$. We first take the natural filtration of $S^k\Cal
AM$, with components indexed from $-k$ to $k$, and then define $\Cal
W^\ell M$ to be the annihilator of the filtration component with
index $-\ell+1$. Explicitly, this means that $\Cal W^\ell M$ to
consist of all maps $\Ps\in\Cal WM$ such that $\Ps(s_1,\dots,s_k)=0$
for arbitrary elements $s_j\in\Cal A^{i_j}M$, provided that
$i_1+\dots+i_k>-\ell$.  Then by definition, we get $\Cal
W^{\ell+1}M\subset\Cal W^\ell M$ for each $\ell$, $\Cal W^{k+1}M=0$,
and $\Cal W^{-k}M=\Cal WM$. Moreover, a map $\Ph\in\Cal W^kM$ by
definition vanishes if at least one of its entries is from $\Cal
A^0M\subset\Cal AM$. Hence this factors to a $k$--linear symmetric map
on copies of $\Cal AM/\Cal A^0M\cong TM$, and we get an isomorphism
$\Cal W^kM\cong S^kT^*M\otimes E$. We will denote by
$\iota:S^kT^*M\otimes E\to\Cal WM$ the corresponding natural
inclusion.

\begin{prop}\label{P3.1}
  (1) The symmetrized $k$--fold fundamental derivative
  $D^{(k)}:\Ga(E)\to\Ga(\Cal WM)$ has values in the space of sections
  of the subbundle $\Cal W^0M$.

\noindent
(2) Consider any principal connection on the bundle $\Cal G_0\to M$,
denote by $\nabla$ all the induced connections on associated vector
bundles, by $\nabla^{k}$ the $k$--fold covariant derivative, and by
$\nabla^{(k)}$ its symmetrization. 

Then the operator $\Ga(E)\to\Ga(\tilde{\Cal W}M)$ given by $\ph\mapsto
D^{k}\ph-i(\nabla^{k}\ph)$ has order at most $k-1$. In particular,
$D^{(k)}\ph$ is the sum of $i(\nabla^{(k)}\ph)$ and terms of order at
most $k-1$ in $\ph$.
\end{prop}
\begin{proof}
  We will proceed by induction on $k$. Recall that there is a family
  of preferred connections on the bundle $\Cal G_0$ which is intrinsic
  to the AHS--structure, see \cites{Weyl,tractors}. Any such connection
  also determines determines a splitting of the filtration of the
  adjoint tractor bundle, i.e.~an isomorphism $\Cal AM\to
  T^*M\oplus\End_0(TM)\oplus TM$, where $\End_0(TM)=\Cal
  G_0\x_{G_0}\fg_0$, which behaves well with respect to the
  filtration. In particular, the last component is given by the
  natural projection $\Pi:\Cal AM\to TM$, while the first component
  restricts to the natural isomorphism $\Cal A^1M\to T^*M$. Fixing one
  preferred connection, the difference to any other principal
  connection on $\Cal G_0$ is given by a tensorial operator, so it
  suffices to prove part (2) for the chosen preferred connection.
  
  A formula for the action of the fundamental derivative on tensor
  bundles in terms of $\nabla$ and this splitting is derived in
  section 4.14 of \cite{tractors}. The argument used there applies to
  all bundles constructed from completely reducible subquotients of
  tractor bundles, and hence to all bundles associated to $\Cal G_0$.
  If $s\in\Ga(\Cal AM)$ corresponds to $(\ps,\Ph,\xi)$ in the
  splitting determined by $\nabla$, then
  $D_s\ph=\nabla_\xi\ph-\Ph\bullet\ph$, where $\bullet:\End_0(TM)\x
  E\to E$ is the tensorial operation induced by the infinitesimal
  action $\fg_0\x V\to V$. Now $s\in\Ga(\Cal A^1M)$ if and only if
  $\xi=0$ and $\Ph=0$, so $D_s\ph=0$ in this case. On the other hand,
  $\xi=\Pi(s)$ so $D_s\ph-\nabla_{\Pi(s)}\ph=\Ph\bullet\ph$ is
  tensorial. Hence we have proved (1) and (2) for $k=1$.

Next observe that naturality of the fundamental derivative implies
that for $s_0,\dots,s_k\in\Ga(\Cal AM)$ we obtain the Leibniz rule
\begin{equation*}\tag{$*$}
  \begin{aligned}
(D^{k+1}\ph)(&s_0,\dots,s_k)=\\
&D_{s_0}(D^k\ph(s_1,\dots,s_k))-
\textstyle\sum_{i=1}^k(D^k\ph)(s_1,\dots,D_{s_0}s_i,\dots,s_k),     
  \end{aligned}
\end{equation*}
compare with proposition 3.1 of \cite{tractors}. Assuming inductively
that part (2) holds for $k$, the second summand is evidently of order
at most $k$ in $\ph$. Moreover, the first summand is given by
$\nabla_{\Pi(s_0)}(\nabla^k\ph(\Pi(s_1),\dots,\Pi(s_k)))$ plus terms of
order at most $k-1$ in $\ph$ which immediately implies (2). 

To prove (1), observe $D^k\ph\in\Ga(\Cal W^0M)$ if and only if
$D^k\ph(s_1,\dots,s_k)=0$ provided that at least $r$ of the sections
$s_i$ have values in $\Cal A^0M$ and at least $k-r+1$ of them even
have values in $\Cal A^1M$. We assume this inductively and prove the
corresponding property of $D^{k+1}\ph$. Hence we take sections
$s_0,\dots,s_k$, and assume that $r'$ of them have values in $\Cal
A^0M$ and $k-r'+2$ even have values in $\Cal A^1M$.

If $s_0$ has values in $\Cal A^1M$, then $D_{s_0}$ acts trivially on
$\Ga(E)$ as well as on sections of $\Cal A^1M$, it maps sections of
$\Cal AM$ to sections of $\Cal A^0M$ and sections of $\Cal A^0M$ to
sections of $\Cal A^1M$. Hence the first summand of the right hand
side of ($*$) vanishes. In the second term of this right hand side,
only summands in which $s_i$ does not have values in $\Cal A^1M$ can
provide a non--zero contribution. If $s_i\in\Ga(\Cal A^0M)$, then in
the corresponding summand we have $r'-1$ sections of $\Cal A^0M$, and
$k-r'+2=k-(r'-1)+1$ of them have values in $\Cal A^1M$, so the
corresponding summand vanishes by inductive hypothesis. If $s_i$ is
not a section of $\Cal A^0M$, then in the corresponding summand we
have $r'$ sections of $\Cal A^0M$, and $k-r'+1$ of them have values in
$\Cal A^1M$, so again vanishing follows by induction.

If $s_0$ has values in $\Cal A^0M$ but not in $\Cal A^1M$, then we
only need to take into account that, acting on sections of $\Cal AM$,
$D_{s_0}$ preserves sections of each filtration component. This shows
that in each of the summands in the right hand side of ($*$), there
are $r'-1$ sections of $\Cal A^0M$ inserted into $D^k\ph$, and
$k-r'+2=k-(r'-1)+1$ of them have values in $\Cal A^1M$. Hence again
vanishing of each summand follows by induction.

Finally, if $s_0$ does not have values in $\Cal A^0M$, then we again
need only that $D_{s_0}$ preserves sections of each of the filtration
components of $\Cal AM$. This shows that in each summand of the right
hand side of ($*$), we have $r'$ sections of $\Cal A^0M$ and $k-r'+2$
of them have values in $\Cal A^1M$. Thus vanishing of each summand
again follows by induction, and the proof of (1) follows by
symmetrization.
\end{proof}

\subsection{The splitting operators}\label{3.2}
According to the the idea described in the beginning of section
\ref{3}, we should next consider the bundle $S^kTM\otimes E^*\otimes
F[\delta]$ of symbols as a quotient of the bundle $\tilde{\Cal
  V}M:=S^k\Cal AM\otimes E^*\otimes F[\delta]$. However, in view of
proposition \ref{P3.1}, we can already improve the basic idea. As we
have noted in \ref{3.1}, the bundle $S^k\Cal AM$ carries a
natural filtration. Taking the tensor product with $E^*$ and
$F[\delta]$, we obtain a filtration of the bundle $\tilde{\Cal V}M$ of
the form
$$
\tilde{\Cal V}^kM\subset\dots\subset\tilde{\Cal
  V}^0M\subset\dots\subset\tilde{\Cal V}^{-k}M=\tilde{\Cal V}M.
$$
As we have observed in the beginning of section \ref{3}, there is a
well defined bilinear pairing $\tilde{\Cal V}M\x\Cal WM\to F[\delta]$.
By definition of the filtration on $\Cal WM$, this factorizes to a
bilinear pairing of $\Cal VM\x\Cal W^0M\to F[\delta]$, where $\Cal
VM:=\tilde{\Cal V}M/\tilde{\Cal V}^1M$. We denote all these pairings
by $\langle\ ,\ \rangle$. As we shall see below, replacing the bundle
$\tilde{\Cal V}M$ by its quotient $\Cal VM$ leads to a smaller set of
critical weights $\delta$. 

For the same reason, it is preferable to take a further decomposition
according to irreducible components of the bundle of symbols as
follows. By construction, the filtration on $S^k \Cal AM$ is induced by
$P$--invariant subspaces of the representation $S^k\fg$, so the
filtration of $\tilde{\Cal V}M$ comes from a $P$--invariant filtration
of $S^k\fg\otimes V^*\otimes W[\delta]$. The quotient of this space by the
largest proper filtration component by construction is $S^k(\frak
g/\frak p)\otimes V^*\otimes W[\delta]$, which induces the bundle of
symbols. Now if we restrict to the subgroup $G_0\subset P$, then $\fg$
decomposes into the direct sum $\fg_{-1}\oplus\fg_0\oplus\fg_1$, and
the filtration components are just $\fg_0\oplus\fg_1$ and
$\fg_1$. Correspondingly, the filtrations on $S^k\fg$ and
$S^k\fg\otimes V^*\otimes W[\delta]$, viewed as $G_0$--representations, are
induced from direct sum decompositions. 

Since we have assumed that $V$ and $W$ are irreducible representations
of $P$ (and hence of $G_0$), the tensor product $S^k(\frak g/\frak
p)\otimes V^*\otimes W[\delta]$ splits into a direct sum $\oplus_i R_i$ of
irreducible representations of $G_0$. Identifying $\fg/\frak p$ with
$\fg_{-1}$, we can view each $R_i$ as a subspace in the quotient of
$S^k\fg\otimes V^*\otimes W$ by the $P$--invariant filtration component
with index 1. Then for each $i$, we can look at the $P$--module $S_i$
generated by $R_i$. Each $S_i$ has a $P$--invariant filtration
with completely reducible subquotients, and the quotient of $S_i$ by
the largest proper filtration component is $R_i$.

Passing to associated bundles, we see that for each $i$, we can
consider $\Cal G\x_P R_i$ as a subbundle of the bundle $S^kTM\otimes
E^*\otimes F[\delta]$ of symbols, and these subbundles form a decomposition
into a direct sum. In particular, any section $\tau$ of the bundle of
symbols can be uniquely written as $\tau=\sum_i\tau_i$ of sections
$\tau_i\in\Ga(\Cal G\x_P R_i)$. Likewise, for each $i$, we can view
$\Cal G\x_PS_i$ as a subbundle of $\Cal VM$, so in particular,
sections of $\Cal G\x_PS_i$ can be viewed as sections of $\Cal VM$. 

Now for each $i$, we denote by $\be^0_i$ the eigenvalue by which the
curved Casimir operator acts on sections of the irreducible bundle
$\Cal G\x_P R_i$, see proposition \ref{P2.4}. Further, by
$\be_i^1,\dots,\be_i^{n_i}$ we denote the different Casimir
eigenvalues occurring for irreducible components in the other
quotients of consecutive filtration components of $S_i$. Using this,
we can now formulate:
 
\begin{prop}\label{P3.2}
  Let $\Pi:\Cal VM\to S^kTM\otimes E^*\otimes F[\delta]$ be
  the natural projection and denote the induced tensorial operator on
  sections by the same symbol. For each $i$ define
  $\ga_i:=\prod_{j=1}^{n_i}(\be^0_i-\be^j_i)$.
  
  Then there is a natural operator $L:\Ga(S^kTM\otimes E^*\otimes
  F[\delta])\to\Ga(\Cal VM)$ such that 
$$
\Pi(L(\tau))=\textstyle\sum_i \ga_i\tau_i
$$
for any section $\tau=\sum_i\tau_i$ of the bundle of symbols.  
\end{prop}
\begin{proof}
Of course for each $i$, mapping $\tau$ to $\tau_i\in\Ga(\Cal
G\x_PR_i)$ defines a tensorial natural operator. The construction of
splitting operators in theorem 2 of \cite{Casimir} gives us, for each
$i$, a natural differential operator $L_i:\Ga(\Cal
G\x_PR_i)\to\Ga(\Cal G\x_PS_i)$. This has the property that denoting
by $\Pi_i$ the tensorial projection in the other direction, we obtain
$\Pi_i(L(\tau_i))=\ga_i\tau_i$ for the number $\ga_i$ defined in the
proposition. As we have noted above, we can naturally view sections of
$\Cal G\x_PS_i$ as sections of $\Cal VM$, so we can simply define
$L(\tau):=\sum_i L_i(\tau_i)$. 
\end{proof}

It is easy to give an explicit description of $L$, since the
construction of splitting operators in \cite{Casimir} is explicit.
Given $\tau$, we have to choose sections $s_i\in\Ga(\Cal
G\x_PS_i)\subset\Ga(\Cal VM)$ such that $\Pi(s_i)=\tau_i$ for all $i$.
Then we claim that 
$$
L(\tau)=\sum_i\prod_{j=1}^{n_i}(\Cal C-\be_i^j)(s_i).
$$ 
The product for fixed $i$ exactly corresponds to the definition of the
splitting operator from \cite{Casimir}. Naturality of the curved
Casimir operator thus implies that each of the summands equals
$L_i(\tau_i)$, viewed as a section of $\Cal VM$, and the claim
follows.

\subsection{The quantization scheme}\label{3.3}
We are now ready to formulate our first main result.
\begin{thm}\label{T3.3}
  The map $(\tau,\ph)\mapsto \langle L(\tau),D^{(k)}\ph\rangle$
  defines a natural bilinear operator $\Ga(S^kTM\otimes E^*\otimes
  F[\delta])\x\Ga(E)\to \Ga(F[\delta])$. 
  
  For $\tau=\sum_i\tau_i\in\Ga(S^kTM\otimes E^*\otimes F[\delta])$,
  the operator $A_\tau:\Ga(E)\to\Ga(F[\delta])$ defined by
  $A_\tau(\ph):=\langle L(\tau),D^{(k)}\ph\rangle$ is of order at most
  $k$ and has principal symbol $\sum_i\ga_i\tau_i$.
\end{thm}
\begin{proof}
  Naturality of $L$, $D^{(k)}$, and the pairing $\langle\ ,\ \rangle$
  implies naturality of the bilinear operator. Now fix $\tau$ and
  consider the operator $A_\tau$. Choose any principal connection on
  $\Cal G_0$ and denote by $\nabla$ all the induced linear connections
  on associated vector bundles. Using proposition \ref{P3.1} we see
  that $A_\tau(\ph)=\langle L(\tau),i(\nabla^{(k)}\ph)\rangle$ up to
  terms of order at mots $k-1$ in $\ph$. Hence $A_\tau$ is of order at
  most $k$ and by the properties of the pairing $\langle\ ,\ \rangle$,
  the principal symbol is obatined as the result of pairing
  $\Pi(L(\tau))\in\Ga(S^kTM\otimes E^*\otimes F[\delta])$ with
  $\nabla^{(k)}\ph\in\Ga(S^kT^*M\otimes E)$. Thus the result follows
  from proposition \ref{P3.2}.
\end{proof}

Now we define a weight $\delta\in\Bbb R$ to be \textit{critical} if at
least one of the $\ga_i$ is zero for the chosen value of $\delta$. For
non--critical weights, our theorem immediately leads to a natural
quantization:

\begin{kor}\label{C3.3}
  If the weight $\delta$ is not critical, then the map $\tau\mapsto
  A_{\sum_i\ga_i^{-1}\tau_i}$ defines a natural quantization for the
  bundles $E$ and $F[\delta]$. 
\end{kor}

We want to emphasize that the naturality result in the corollary in
particular implies that in the case of the homogeneous model $G/P$ of
the AHS--structure in question the quantization is equivariant (as a
bilinear map) under the natural $G$--action on the spaces of sections
of the bundles in question (which are homogeneous vector bundles in
this case). We can restrict the quantization to the big Schubert cell
in $G/P$, which is diffeomorphic to $\Bbb R^n$, $n=\dim(G/P)$. The
$G$--equivariancy on $G/P$ immediately implies that the result is
equivariant for the Lie subalgebra of vector fields on $\Bbb R^n$
formed by the fundamental vector fields for this $G$--action. Hence
our quantization will specialize to an equivariant quantization in the
usual sense. 

\subsection{The set of critical weights}\label{3.4}
To complete our results, we have to prove that for any choice of
bundles $E$ and $F$ and any order $k$, the set of critical weights is
finite. Verifying this is a question of finite dimensional
representation theory. In fact, we not only get an abstract proof of
finiteness of the set of critical weights, but a method to determine
the set of critical weights for any given example.

In view of proposition \ref{P3.2} and theorem \ref{T3.3}, it is clear
that we have to understand the dependence of the Casimir eigenvalues,
or more precisely of the differences $\be^0_i-\be^j_i$, on $\delta$.
To get a complete understanding of the set of critical weights, one
has to determine the composition series (i.e.~the structure of the
quotients of iterated filtration components), of the $P$--modules
$S_i$. Recall from \ref{3.2} that, as a representation of $G_0$, $S_i$
is simply the direct sum of all the composition factors, so
essentially we have to determine the decomposition of $S_i$ into
irreducible components as a $G_0$--module. From proposition \ref{P2.4}
we know how to determine the numbers $\be$ from the lowest weights of
these irreducible components. Notice that changing the weight $\delta$
corresponds to taking a tensor product with a one--dimensional
representation. In particular, this does not influence the basic
decompositions into irreducible components, apart from the fact that
each of these components is tensorized with that one--dimensional
representation. As we shall see, we can get quite a bit of information
without detailed knowledge of the decomposition into irreducibles,
using only structural information on the possible irreducible
components. We start by proving a basic finiteness result.

\begin{thm}\label{T3.4}
  Fix an irreducible component $R_i\subset S^k\fg_{-1}\otimes V^*\otimes
  W[\delta]$, consider the corresponding Casimir eigenvalue $\be_i^0$,
  and one of the other Casimir eigenvalues $\be_i^j$. Then there is
  exactly one value of $\delta$ for which $\be_i^0=\be_i^j$. Hence
  there are at most $n_i$ many values for $\delta$ for which
  $\ga_i=0$, and at most $\sum_in_i$ critical weights. 
\end{thm}
\begin{proof}
  Let us first make a few comments. The Casimir eigenvalues can be
  computed from lowest weights, which are defined via complexification
  of non--complex representations and of the Lie algebra in question.
  Since these complexifications do not change the decomposition into
  irreducible components, we may work in the setting of complex
  $|1|$--graded Lie algebras throughout the proof. Second, recall that
  for an irreducible representation of a complex semisimple Lie
  algebra, the negative of the lowest weight coincides with the
  highest weight of the dual representation. In this way, standard
  results on highest weights have analogs for the negatives of lowest
  weights.
  
  As we have noted in proposition \ref{P2.4}, for a representation
  with lowest weight $-\la$, the Casimir eigenvalue on sections of the
  corresponding induced bundle is given by
  $\|\la\|^2+2\langle\la,\rho\rangle=\langle\la,\la+2\rho\rangle$.
  Writing $c_\la$ for this number, the last expression immediately
  shows that for two weights $\la$ and $\la'$, we have
  \begin{equation}
    \label{cdiff}
c_{\la'}-c_\la=2\langle \la'-\la,\la+\rho\rangle+\|\la'-\la\|^2.     
  \end{equation}
  We have to understand, how this is influenced by changing $\delta$.
  Denoting by $\mu$ the highest weight associated to the
  representation $\Bbb R[1]$, which induces the bundle $\Cal E[1]$,
  the bundle $\Cal E[w]$ corresponds to the weight $w\mu$. Moving from
  $\delta$ to $\delta+w$ corresponds to forming a tensor product with
  $\Cal E[w]$, and hence replacing $\la$ by $\la+w\mu$ and $\la'$ by
  $\la'+w\mu$. This means that the difference of the two weights
  remains unchanged, and equation \eqref{cdiff} shows that
  \begin{equation}
    \label{crit}
c_{\la'+w\mu}-c_{\la+w\mu}=c_{\la'}-c_\la+2w\langle
\la'-\la,\mu\rangle. 
  \end{equation}
  Now by definition, the weights of the representation $\fg$ are
  exactly the roots of $\fg$. Consequently, any weight of $S^k\fg$ is
  a sum of $k$ roots. Further, it is well know that the highest weight
  of any irreducible component in a tensor product of two irreducible
  representations can be written as a sum of the highest weight of one
  of the two factors and some weight of the other factor. Passing to
  duals, we see that the same statement holds for the negatives of
  lowest weights. Thus, the negative of the lowest weight of any
  irreducible component of $S^k\fg\otimes V^*\otimes W$ can be written
  as a linear combination of the negative of the lowest weight of an
  irreducible component of $V^*\otimes W$ and at most $k$ roots.
  
  Now recall (see \cite{Yamaguchi}) that for a complex $|1|$--graded
  Lie algebra, one can choose a Cartan subalgebra $\frak h\subset\fg$
  and positive roots in such a way that there is a unique simple root
  $\al_0$ for which the corresponding root space is contained in
  $\fg_1$. More precisely, for a root $\al$, the corresponding root
  space sits in $\fg_i$ for $i={-1,0,1}$, where $i$ is the coefficient
  of $\al_0$ in the expansion of $\al$ as a sum of simple roots. Then
  the center of $\fg_0$ is generated by the unique element
  $H_0\in\frak h$ for which that $\al_0(H_0)=1$, while all other
  simple roots vanish on $H_0$. The orthocomplement of $H_0$ in $\frak
  h$ is a Cartan subalgebra of the semisimple part of $\fg_0$.
  
  Since the semisimple part of $\fg_0$ acts trivially on $\Bbb R[1]$,
  we conclude that $\mu(H)=aB(H,H_0)$ for some nonzero number $a$ and
  all $H\in\frak h$, where $B$ denotes the Killing form of $\fg$.
  Going through the conventions, it is easy to see that actually
  $a<0$. By definition of the inner product, this means that for any
  weight $\nu$, we have $\langle \mu,\nu\rangle=a\nu(H_0)$. Since
  $H_0$ acts by a scalar on any irreducible representation, it also
  acts by a scalar on all of $S^k\fg_{-1}\otimes V^*\otimes W$. But
  this implies that if $-\nu$ is the lowest weight of an irreducible
  component of $V^*\otimes W$, then $\nu(H_0)=a_0$ for a fixed number
  $a_0$.  Consequently, if $-\nu$ is the lowest weight of an
  irreducible component of the quotient of two consecutive filtrations
  components in $\Cal VM$, say the one with index $\ell$ by the one
  with index $\ell+1$, $\nu(H_0)=a_0+\ell$. In particular, if $-\la$
  is the lowest weight of $R_i$, then $\la(H_0)=a_0-k$. Likewise if
  $-\la'$ is the lowest weight giving rise to $\be_i^j$ then
  $\la'(H_0)=a_0+\ell$ for some $\ell>-k$. Thus we conclude that
  $\langle\la'-\la,\mu\rangle=a(k+\ell)<0$, and formula
  \eqref{crit} shows that $\la$ and $\la'$ give rise to exactly one
  critical weight.
\end{proof}

Note that the proof actually leads to an explicit formulae for the
critical weights. Suppose that $-\la$ and $-\la'$ are the lowest
weights of irreducible components giving rise to $\be_i^0$ and
$\be_i^j$, and that the irreducible component corresponding to $-\la'$
sits in the quotient of the $\ell$th by the $(\ell+1)$st filtration
component. Then formulae \eqref{cdiff} and \eqref{crit} from the proof
show that the critical weight caused by these two components is given
by
\begin{equation}
  \label{critweight}
\delta=\frac{2\langle\la'-\la,\la+\rho\rangle+\|\la'-\la\|^2}
{2\langle\la'-\la,\mu\rangle}
\end{equation}
where $\mu$ is the highest weight of the representation $\Bbb R[1]$.
In particular, we can use this formula to completely determine the set
of all critical weights if we know all the $P$--representations $S_i$
together with their composition structure.

\subsection{Restrictions on critical weights}\label{3.5}
We can also get some information on the set of critical weights
without this detailed knowledge. For any $P$--module, we can look at
the restriction of the $P$--action to $G_0$ and the restriction of the
infinitesimal action of $\frak p$ to the abelian subalgebra $\fg_1$.
Since $P$ is the semidirect product of $G_0$ and $\exp(\fg_1)$,
one immediately concludes that any subspace in a representation of
$P$, which is $G_0$--invariant and closed under the infinitesimal
action of $\fg_1$ is actually $P$--invariant. By construction, the
actions of elements of $\fg_1$ on any $P$--module commute. Hence
the iterated action of elements of $\frak g_1$ (in the $P$--module
$S^k\fg\otimes V^*\otimes W[\delta]$) on $R_i$ define maps
$S^\ell\fg_1\otimes R_i\to S^k\fg \otimes V^*\otimes W[\delta]$. By
construction, the image sits in the filtration component with index
$\ell-k$ as well as in $S_i$. Hence we actually obtain a map
$\oplus_{\ell=0}^kS^\ell\fg_1\otimes R_i\to S_i$, which is evidently
$G_0$--equivariant. In particular, the image is a $G_0$--invariant
subspace of $S_i$ and from the construction it follows immediately
that it is also closed under the infinitesimal action of $\fg_1$.

The upshot of this is that any $G_0$--irreducible component of $S_i$
also occurs in $\oplus_{\ell=0}^kS^\ell\fg_1\otimes R_i$. If we
determine the set of all weights $\delta$ for which an irreducible
component of $\oplus_{\ell=1}^kS^\ell\fg_1\otimes R_i$ corresponds to
the same Casimir eigenvalue as $R_i$, then the union of these sets for
all $i$ contains the set of all critical weights. 

We next work out more details on the set of critical weights for some
examples in the case of even dimensional conformal structures of
arbitrary signature $(p,q)$. (This is significantly more complicated
than the case of projective structures, which is mainly considered in
the literature). Hence $G_0$ is the conformal group $CO(p,q)$ and
$\fg_{-1}$ is the standard representation $\Bbb R^n$, $n=p+q$ of this
group, and we assume that $n$ is even. As above, we may work in the
complexified setting, and we will use the notation, conventions and
results from \cite{CasEx} for weights. We will fix representations $V$
and $W$ and determine critical weights starting from
$S^k\fg_{-1}\otimes V^*\otimes W$ (i.e.~with $\delta=0$).

Let us assume that $S^k\fg_{-1}\otimes V^*\otimes W$ contains an
irreducible component $R_i\cong \Bbb R[w]$ for some $w\in\Bbb R$. The
decomposition of $S^\ell\Bbb R^{n*}$ into irreducible components is
given by $S^\ell_0\Bbb R^{n*}\oplus S^{\ell-2}_0\Bbb R^{n*}[-2]\oplus
S^{\ell-4}_0\Bbb R^{n*}[-4]\oplus\dots$, where the subscript $0$
indicates the totally tracefree part. From \ref{3.5} we thus conclude
that in any case all the irreducible components of $P$--module $S_i$
generated by $R_i$ must be of the form $S^\ell\Bbb R^{n*}[w-2m]$ for
non--negative integers $\ell$ and $m$ such that $\ell+2m\leq k$.

In particular, for $k=1$, the only possibility is $\Bbb R^n[w]$. In
the notation from section 2.4 of \cite{CasEx}, $\Bbb R[w]$ corresponds
to the weight $(w|0,\dots)$ while $\Bbb R^{n*}[w]$ corresponds to
$(w-1|1,0,\dots)$, which immediately shows that the corresponding
critical weight is $\delta=-w$. For $k=2$, we get $S^2_0\Bbb
R^{n*}[w]$ and $\Bbb R[w-2]$, which correspond to $(w-2|2,0,\dots)$
and $(w-2|0,\dots)$ and the critical weights $1-w$ and
$1-w-\tfrac{n}2$.

For a general order $k$, the possible representations are
$(w-\ell|\ell-2m,0,\dots)$ for $\ell\leq k$ and $\ell-2m\geq 0$ and
one easily verifies directly:
\begin{prop}\label{P3.5}
  The possible critical weights caused by an irreducible component
  $\Bbb R[w]\subset S^k\fg_{-1}\otimes V^*\otimes W$ are contained in
  the set 
$$
\left\{-w-1+\ell-2m+\frac{m(2+2m-n)}{\ell}:0\leq\ell\leq k,0\leq
  2m\leq\ell\right\}.
$$ 
\end{prop}

We can derive an effective upper bound, above which there are no
critical weights for quantization in \textit{any} order. This can be
viewed as a vast generalization of the results in section 3.1 of
\cite{DLO} on quantization of operators on functions. Observe first
that it may happen that for the representations $V$ and $W$ inducing
$E$ and $F$, the tensor product $V^*\otimes W$ itself splits into
several irreducible components. For example, if $V=W$, then one always
has the one dimensional invariant subspace spanned by the identity.
Given an irreducible component $U\subset V^*\otimes W$ and
$\delta\in\Bbb R$, we have $S^k\fg_{-1}\otimes U[\delta]\subset
S^k\fg_{-1}\otimes V^*\otimes W[\delta]$, so one may talk about
symbols of type $U$ of any order and any weight. Of course, one may
apply the constructions from \ref{3.1}--\ref{3.3} directly to this
subspace.  As an irreducible representation of $\fg_0$, $U[\delta]$
has an associated lowest weight. Using this, we can now formulate

\begin{thm}\label{T3.5}
  Let $-\la$ be the lowest weight of $U[\delta]$ and assume that
  $\delta$ is chosen in such a way that $\la$ is $\fg$--dominant. Then
  for any order $k$, the weight $\delta$ is non--critical for symbols
  of type $U$. In particular, this always holds for sufficiently large
  values of $\delta$. 
\end{thm}
\begin{proof}
  Let us first assume that $\la$ is $\fg$--dominant and integral. Then
  there is a finite dimensional irreducible representation $\tilde U$
  of $\fg$ with lowest weight $-\la$. We can pass to the dual $\tilde
  U^*$, and look at the $\frak p$--submodule generated by a highest
  weight vector. It is well known that this realizes the irreducible
  representation of $\frak p$ with highest weight $\la$. Passing back,
  we see that $U[\delta]$ can be naturally viewed as a quotient of
  $\tilde U$.  Consequently, for any $k\geq 0$, we can naturally view
  $S^k\fg_{-1}\otimes U[\delta]$ as quotient of the representation
  $S^k\fg\otimes\tilde U$ of $\fg$.  In particular, for any
  irreducible component $R_i\subset S^k\fg_{-1}\otimes U[\delta]$ we
  obtain a corresponding $\fg$--invariant subset $\tilde S_i\subset
  S^k\fg\otimes\tilde U$ (which can be taken to be $\frak
  g$--irreducible) with $\frak p$--irreducible quotient $R_i$. It is
  also evident that applying the natural map $S^k\fg\otimes\tilde U\to
  S^k\fg\otimes U[\delta]$ to $\tilde S_i$ and then factoring by the
  filtration component of degree zero, the image has to contain the
  $\frak p$--submodule $S_i$ generated by $R_i$. In particular, any
  $\fg_0$--irreducible component of $S_i$ also has to occur in $\tilde
  S_i$.
  
  But for the bundles corresponding to irreducible representations of
  $\frak g$, the critical weights are described in lemma 2 of
  \cite{Casimir} in terms of the Kostant Laplacian $\square$ and the
  value $c_0$ by which the (algebraic) Casimir operator of $\fg$ acts
  on the irreducible representation $\tilde S_i$. Now $c_0$ coincides
  with the Casimir eigenvalue $\be^0_i$ in our sense and hence lemma 2
  of \cite{Casimir} shows that $\be_i^j-\be^0_i$ can be computed as
  twice the eigenvalue of $\square$ on the irreducible component
  giving rise to $\be^i_j$. Now Kostant's theorem from \cite{Kostant}
  in particular implies that the kernel of $\square$ on $\tilde S_i$
  consists of $R_i$ (viewed as a $\fg_0$--invariant subspace) only.
  This implies the result if $\la$ is $\fg$--dominant and integral.
  
  More is known about the eigenvalues of $\square$, however. The lemma
  in Cartier's remarks (\cite{Cartier}) to Kostant's article shows
  that all eigenvalues of square are non--positive. In the terminology
  of the proof of theorem \ref{T3.4} this means that
  $c_{\la'}-c_\la<0$. There we have also seen that
  $\langle\la'-\la,\mu\rangle<0$, so formula \eqref{crit} from that
  proof shows that $c_{\la'+w\mu}-c_{\la+w\mu}<0$ for $w\geq 0$. Now
  if $-\la$ is the lowest weight of a finite dimensional irreducible
  representation of $\frak p$, then $\la$ is $\frak p$--dominant and
  $\frak p$--integral. But this means that $\la+w\mu$ is $\frak
  g$--dominant for sufficiently large values of $w$ and $\frak
  g$--integral for all integral values of $w$, which implies all the
  remaining claims.   
\end{proof}

% \noindent
% {\bfseries Remark/Theorem:}
% \idx{(A more general version of Theorem \ref{T3.5})
% Let $U[\delta]$ satisfies the assumption of Theorem \ref{T3.5}, i.e.\
% $\delta$ is non-critical for symbols of type $U$. 
% Then $U[\delta] \otimes E[w]$ is non-critical for every $w \geq 0$.}
% \begin{proof}
% Using the notation from the proof, we have $c_{\la'}-c_{\la}<0$ for
% $\la' \not= \la$ (the Cartier's lemma from the Kostant's article).  We
% show $\langle \la'-\la, \mu \rangle = a(k+\ell)$ and the right hand
% side always negative. (Here $k+\ell>0$ and to my understanding $a<0$ -
% is it the correct inequality here ?) But then $w \langle \la'-\la, \mu
% \rangle <0$ for $w>0$ thus (\ref{crit}) is negative for all weights
% $\la' \not= \la$. In particular, (\ref{crit}) is always nonzero.
% \end{proof}
% 
% \vspace{1ex}
% 
% Note that for any irreducible representation $U$ of $\fg_0$ with
% lowest weight $-\la$, the inner product of $\la$ with all simple roots
% except the root $\al_0$ from the proof of theorem \ref{T3.4} is a
% non--negative integer. Since from that proof we see that
% $\langle\la,\al_0\rangle$ depends linearly on $\delta$, we conclude
% that there always are infinitely many values of $\delta$ to which our
% theorem applies.

\subsection{Low order quantizations for even--dimensional conformal
  structures}\label{3.6}
Let us move to more complete examples in the setting from above. We
will restrict to the cases that $V^*\otimes W\cong\Bbb R$ and $V^*\otimes
W\cong\Bbb R^n$, and to orders at most three in the first case and at
most two in the second case. For $V^*\otimes W\cong\Bbb R$, we get
quantizations on density bundles, which can be compared to available
results in the literature. The case $V^*\otimes W\cong\Bbb R^n$ can be
used to understand operators mapping weighted one--forms to densities
and, vice versa, mapping densities to weighted one--forms. 

We have already noted in \ref{3.5} that the decomposition of
$S^k\fg_{-1}$ is given by $\oplus_{\ell\leq
  k/2}S^\ell\fg_{-1}[2\ell]$. 

\noindent
\textbf{First order operators on densities}. Here the symbol
representation is $\fg_{-1}\cong\Bbb R^n$, so this is irreducible and
corresponds to the weight $(1|1,0\dots)$. Likewise, $\fg$ is an
irreducible representation of $\fg$, and there is only one relevant
level which may produce critical weights, namely
$\fg_0\cong\La^2\Bbb R^n[2]\oplus\Bbb R$, which is the quotient of the
filtration components of degrees 0 and 1. The summands correspond to
the weights $(0|1,1,0,\dots)\oplus (0|0,\dots)$ and we obtain the
critical weights $-n$ and $-2$. 

\smallskip
\noindent
\textbf{Second order operators on densities}. The symbol
representation splits into two irreducible components $R_1$ and $R_2$
corresponding to the weights $(2|2,0,\dots)$ (tracefree symbols) and
$(2|0,\dots)$ (symbols which are pure trace, i.e.~of Laplace type).
Also, the representation $S^2\fg$ of $\fg$ is not irreducible any
more, but splits into four irreducible components. One of them is a
trivial representation (corresponding to the Killing form) and one is
isomorphic to $\La^4\Bbb R^{n+2}$. These two components are entirely
contained in the filtration component of degree $-1$, so they do not
contribute to the quotient by the largest filtration component. One of
the remaining two irreducible components is isomorphic to $S^2_0\Bbb
R^{n+2}$. The quotient of this component by its intersection with the
largest filtration component is exactly $R_2$, so all of $S_2$ must be
contained in this part. Finally, there is the highest weight component
$\circledcirc^2\fg\subset S^2\fg$ (the Cartan product of two copies of
$\fg$), whose quotient by the largest filtration component is $R_1$.
Hence $S_1$ is contained in this component.

To determine the possible critical weights it thus suffices to
analyze the composition structure of the representations
$\circledcirc^2\fg$ and $S^2_0\Bbb R^{n+2}$. This can be done fairly
easily using the description of representations of $\fg$ in terms of
their $\frak p$--irreducible quotients from section 3 of \cite{BCEG},
in particular the result in lemma 3.1 of this article. One has to use
the fact that the Lie algebra cohomology groups that occur are
algorithmically computable using Kostant's version of the
Bott--Borel--Weil theorem. 

This shows that in the language of weights, the two relevant levels of
$\circledcirc^2\fg$ decompose as 
$$
\begin{aligned}
  &(1|2,1,0,\dots)\oplus (1|1,0,\dots)\\
  &(0|2,2,0,\dots)\oplus (0|2,0,\dots)\oplus (0|1,1,0,\dots)\oplus
  (0|0,\dots),
\end{aligned}
$$
and consequently, one obtains the critical weights $-3$, $-2$, $-2-n$,
$-1-n$, $(-2-n)/2$, and $(-4-n)/2$. 

For the case of symbols which are pure trace, the decompositions of
the level for the index $-1$ is irreducible corresponding to the weight
$(1|1,0\dots)$, while the level for index zero decomposes as
$(0|2,0,\dots)\oplus (0|0,\dots)$. This gives rise to the critical
weights $-2$, $-1$ and $(-2-n)/2$.

\smallskip

\noindent
\textbf{Third order operators on densities}. The analysis is closely
analogous to the second order case, we mainly include the results for
comparison to \cite{3rd-order}. The symbol representation splits into
two irreducible components and again these two components correspond
to two of the seven irreducible components in $S^3\fg$. Namely,
tracefree symbols ($S^3_0\Bbb R^n$) correspond to the highest weight
component $\circledcirc^3\fg$, while trace--symbols ($\Bbb R^n[2]$)
correspond to the Cartan product $\fg\circledcirc S^2_0\Bbb R^{n+2}$.
The relevant parts of the composition series for these two
representations of $\fg$ can be determined as in the second order
case. From these, one computes the critical weights. In the tracefree
case, one obtains $-4$, $-3$, $-2$, $-4-n$, $-3-n$, $-2-n$,
$(-7-n)/2$, $(-4-n)/2$, $(-8-n)/3$, $(-8-2n)/3$, $(-6-n)/3$, and
$(-6-2n)/3$. For trace--type symbols, we get the critical weights
$-1$, $-2$, $-4$, $-5/2$, $-4/3$, $(-4-n)/2$, $(-4-n)/3$, $(-6-n)/3$,
and $(-4-2n)/3$.  These are the critical weights from
\cite{3rd-order}, plus quite a few additional ones. We'll comment on
that in \ref{3.7} below.

\smallskip

\noindent
\textbf{First order operators for $V^*\otimes W\cong\Bbb R^n$}. Here
the symbol representation decomposes as 
$$
\Bbb R^n\otimes\Bbb R^n=R_1\oplus R_2\oplus R_3=S^2_0\Bbb R^n\oplus
\La^2\Bbb R^n\oplus\Bbb R[2],
$$
or in weights $(2|2,0,\dots)\oplus (2|1,1,0,\dots)\oplus (2|0,\dots)$.
There is only one relevant level in the composition series of
$\fg\otimes\Bbb R^n$, which can be determined by decomposing the
tensor product $\fg_0\otimes\Bbb R^n$ into irreducibles. In terms of
weights, the result is $(1|2,1,0,\dots)\oplus (1|1,1,1,0,\dots)\oplus
2(1|1,0\dots)$, so the last irreducible component occurs with
multiplicity two. Decomposing the tensor products $R_i\otimes\Bbb
R^n$, one concludes that $S_1$ can only contain the first and a copy
of the last irreducible components, while $S_3$ can only contain one
copy of the last irreducible component. Consequently, there are three
critical weights for skew symmetric symbols (which turn out to be
$-1$, $-4$, and $-n$) but only two (namely $-3$ and $-2-n$) for
symmetric symbols. For trace type symbols we obtain only one critical
weight, namely $-2$, which agrees with the result from \ref{3.5}.

\smallskip

\noindent
\textbf{Second order operators for $V^*\otimes W\cong\Bbb R^n$}. Here
the symbol representation $S^2\Bbb R^n\otimes\Bbb R^n$ decomposes into
four irreducible components, in weight notation, it is given by 
$$
(3|3,0,\dots)\oplus (3|2,1,0,\dots)\oplus 2(3|1,0,\dots). 
$$ 
Here one of the two copies of $\Bbb R^n[2]$ is contained in $S^2_0\Bbb
R^n\otimes\Bbb R^n$, while the other comes from the trace part. Let us
write this decomposition as $R_1\oplus\dots\oplus R_4$, with $R_4$
coming from the trace part. From above, we know that $S^2\fg$ contains
the irreducible components $\circledcirc^2\fg$ and $S^2_0\Bbb
R^{n+2}$, which correspond to $S^2_0\Bbb R^n$ and $\Bbb R[2]\subset
S^2\Bbb R^n$, respectively. Consequently, we can determine the
relevant composition factors for $S_1$, $S_2$, and $S_3$ by
decomposing the tensor products of the composition factors of
$\circledcirc^2\fg$ as listed above with $\Bbb R^n$, and then checking
with of the components may be contained in each $S_i$. For $S_4$, we
proceed similarly with $S^2_0\Bbb R^{n+2}$ replacing
$\circledcirc^2\fg$. 

For the first relevant level (corresponding to filtration index $-1$),
we first have to decompose $(1|2,1,0,\dots)\otimes (1|1,0,\dots)$
which gives
$$
(2|3,1,0,\dots)\oplus (2|2,2,0,\dots)\oplus (2|2,1,1,0,\dots)\oplus 
(2|2,0,\dots)\oplus (2|1,1,0,\dots). 
$$
Second, $(1|1,0,\dots)\otimes (1|1,0,\dots)\cong (2|2,0,\dots)\oplus
(2|1,1,0,\dots)\oplus (2|0,\dots)$.

Looking at the tensor products $R_i\otimes\Bbb R^n$, we conclude that
$S_1$ can only contain $(2|3,1,0,\dots)$ and $(2|2,0,\dots)$, $S_3$
can only contain $(2|2,0,\dots)$ and $(2|1,1,0,\dots)$, while all components
of the first sum may occur in $S_2$.  Hence from this level, we get
the critical weights $-4$ and $-4-n$ for $R_1$. For $R_2$, we obtain
the critical weights $-1$, $-3$, $-5$, $-1-n$, and $-3-n$, while for
$R_3$, the critical weights are $-2$, $-4$, and $-2-n$.

The second relevant level is dealt with in an analogous way. The
result is that for $R_1$, we get the additional critical weights $-3$,
$-3-n$, $(-4-n)/2$, and $(-7-n)/2$. For $R_2$, we obtain $-3/2$,
$-7/2$, $(-1-n)/2$, $(-4-n)/2$, $(-7-n)/2$, $(-3-2n)/2$. Finally, for
$R_3$, we get the additional critical weights $-1$, $-5/2$, and
$(-4-n)/2$. A direct evaluation shows that for $R_4$ we exactly the
same critical weights as for $R_3$ (although the bundle involved is
different).

\subsection{Discussion and Remarks}\label{3.7}
(1) Note that the results in the examples from \ref{3.6} are
consistent with theorem \ref{T3.5}, which implies that in all the
cases discussed in \ref{3.6} all critical weights have to be negative.

\smallskip

\noindent
(2) From the examples of operators on densities discussed in \ref{3.6} it
is evident that the sets of critical weights we obtain with our
general procedure are far from being optimal. It is actually easy to
see why this happens, and even to partly improve the procedure, to get
smaller sets of critical weights. The point here is that part (1) of
proposition \ref{P3.1} can be heavily improved in special cases, and in
particular for the fundamental derivative on densities. In the case of
densities, already the values of a single fundamental derivative do not
exhaust $\Cal A^0M[w]$. On the contrary, projecting to $(\Cal A^0M/\Cal
A^1M)[w]\cong \La^2TM[w-2]\oplus\Cal E[w]$, the values always lie in
the density summand only. By naturality of the fundamental derivative,
this implies that higher order fundamental derivatives always will lie
in subbundles which are much smaller than the bundle $\Cal W^0M$ from
proposition \ref{P3.1}. 

Knowing this, one can run the analog of the procedure from \ref{3.2}
and \ref{3.3} on the quotient by the annihilator of this subbundle,
which will be significantly smaller than the bundle $\Cal VM$ we have
used. For this smaller quotient, there will be less irreducible
components in the individual subquotients and hence less critical
weights. In fact, it is easy to see directly that in the examples
discussed in \ref{3.6} most (but not all of) the superfluous critical
weights will disappear. 

\smallskip

\noindent
(3) In the case $V\otimes W\cong\Bbb R^n$ the set of critical weights
we have obtained in \ref{3.6} will be closer to the optimum than in
the case of densities. As we have noted, this case can be used to
study both quantizations for operators mapping sections of $\Cal E[w]$
to sections of $TM[w+\delta]$ and for operators mapping sections of
$T^*M[w]\cong TM[w-2]$ to sections of $\Cal E[w+\delta]$. While these
two cases are completely symmetric from our point of view, this is no
more true if one looks at the best possible sets of critical weights.
The point is that in the first case, the value of the splitting
operator will be paired with $D^{(k)}f\in\Ga(S^k\Cal A^*M[w])$ for
$f\in\Ga(\Cal E[w])$, and as discussed above, this has values in a
much smaller subbundle than just the filtration component of degree
zero. In the second interpretation, we will have to pair it with
$D^{(k)}\al\in \Ga(S^k\Cal A^*M\otimes T^*M[w])$ for
$\al\in\Ga(T^*M[w])$, and the values of this operator fill a more
substantial part of the filtration component of degree zero. Hence in
the first case, we can remove more superfluous critical weights than
in the second one.

\smallskip

\noindent
(4) There is a systematic way to derive explicit formulae for the
procedures we have developed in terms of distinguished connections
(e.g.~the Levi--Civita connections of the metrics in a conformal
class), but this becomes quickly rather tedious. In view of the
construction, the main point is to obtain an explicit formula for the
curved Casimir operator on irreducible components of $S^k\Cal AM$.
This can be done along the lines of proposition 2.2 of \cite{CasEx}
which holds (with obvious modifications) for general AHS--structures.

\end{document}